\documentclass{amsart}
\usepackage{amssymb,amsmath,amscd,xy,graphicx,textcomp}
\newtheorem{theorem}{Theorem}[section]
\newtheorem{lemma}[theorem]{Lemma}
\newtheorem{corollary}[theorem]{Corollary}
\newtheorem{definition}[theorem]{Definition}

\newtheorem{remark}[theorem]{\it Remark}
\newtheorem{example}[theorem]{Example}
\newtheorem{proposition}[theorem]{Proposition}

\xyoption{arrow}

\xyoption{matrix}

\def\C{\mathbb{C}}
\def\R{\mathbb{R}}
\def\Z{\mathbb{Z}}

\def\tree{\mathcal{T}}

\def\ql{\backslash \! \backslash}

\title[Coordinate rings of moduli and toric fiber products]{Coordinate rings for the moduli stack of $SL_2(\C)$ quasi-parabolic principal bundles on a curve and toric fiber products}
\author{Christopher Manon}
\thanks{This work was supported by the NSF fellowship DMS-0902710}

\begin{document}

\begin{abstract}
We continue the program started in \cite{M1} to understand the combinatorial commutative algebra of the projective coordinate rings of the moduli stack $\mathcal{M}_{C, \vec{p}}(SL_2(\C))$ of quasi-parabolic $SL_2(\C)$ principal bundles on a generic marked projective curve. We find general bounds on the degrees of polynomials needed to present these algebras by studying their toric degenerations. In particular, we show that the square of any effective line bundle on this moduli stack yields a Koszul projective coordinate ring.  This leads us to formalize the properties of the polytopes used in proving our results by constructing a category of polytopes with term-orders.  We show that many of results on the projective coordinate rings of $\mathcal{M}_{C, \vec{p}}(SL_2(\C))$ follow from closure properties of this category with respect to fiber products.
\end{abstract}

\maketitle

\tableofcontents

\smallskip

\section{Introduction}

We wish to understand the structure of the projective coordinate rings of the moduli stack $\mathcal{M}_{C, \vec{p}}(G)$ of quasi-parabolic principal $G-$bundles on a marked projective curve $(C, \vec{p}) \in \mathcal{M}_{g, n},$ where $G$ is a simple complex group and the parabolic structure is given by a Borel subgroup $B \subset G.$  Our interest in these objects stems
both from the classical value of problems in the moduli of bundles, and because the graded components of these algebras can be identified with spaces of conformal blocks.    
Conformal blocks for the conformal field theory defined by a simple complex Lie algebra $\mathfrak{g}$ and a non-negative integer $L$ for marked projective complex curves $(C, \vec{p})\in \bar{\mathcal{M}}_{g, n}$ occupy an interesting position in algebraic geometry and mathematical physics.  When the curve is allowed to vary, they form vector bundles over $\bar{\mathcal{M}}_{g, n}$ which have been the object of interesting recent work, \cite{F}, \cite{AGS}.  When the genus of the curve is set to $0,$  they are the structure spaces for a category of representations of the specialization of a quantum group at a root of unity. As we will mention below, their combinatorics have even made appearances in mathematical biology.  Because of the variety of applications,
we seek to understand structural features of conformal blocks, and relate them to the commutative algebra of $\mathcal{M}_{C, \vec{p}}(G)$.

The moduli stack $\mathcal{M}_{C, \vec{p}}(G)$ can be expressed as the quotient stack of a product of the affine Grassmannian variety $\mathcal{Q}$ with the projective variety $[G/B]^n,$ by an action of an ind-group $L_{C, \vec{p}}$ determined by the points $p_1, \ldots, p_n.$

\begin{equation}
\mathcal{M}_{C, \vec{p}}(G) = L_{C, \vec{p}}(G) \ql [\mathcal{Q}\times G/B^n]\\
\end{equation}

\noindent
The Picard group of $\mathcal{M}_{C, \vec{p}}(G),$ calculated in \cite{LS}, is a product of $n$ copies of the character group of $B$ times a copy of $\Z.$

\begin{equation}
Pic(\mathcal{M}_{C, \vec{p}}(G)) = \mathcal{X}(B)^n \times \Z\\
\end{equation}

\noindent
The cone of line bundles with non-zero global sections is a subcone
of  $\Delta^n \times \Z_{\geq 0} \subset \mathcal{X}(B)^n\times \Z,$ where
$\Delta$ is the Weyl chamber of $G.$    The space of global sections $H^0(\mathcal{M}_{C, \vec{p}}(G), \mathcal{L}(\vec{\lambda}, L))$
for a vector of dominant weights $\vec{\lambda}$ of $G,$ and a non-negative integer
$L$ agrees with the space $V_{C, \vec{p}}(\vec{\lambda}, L)$ of conformal blocks of the rational conformal field theory defined by $Lie(G)$ and $L,$ with weights $\vec{\lambda}$ at the marked points $\vec{p}$ and level $L.$   Let $R_{C, \vec{p}}(\vec{\lambda}, L)$ be the projective coordinate ring of $\mathcal{M}_{C, \vec{p}}(G)$ defined by $\mathcal{L}(\vec{\lambda}, L).$

\begin{equation}
R_{C, \vec{p}}(\vec{\lambda}, L) = \bigoplus_{N = 0}^{\infty} H^0(\mathcal{M}_{C, \vec{p}}(G), \mathcal{L}(\vec{\lambda}, L)^{\otimes N})\\
\end{equation}

The Hilbert function of this algebra outputs the sequence of dimensions of the spaces of conformal blocks associated to
$(N\vec{\lambda}, NL),$ which can be calculated by the Verlinde formula from conformal field theory, see \cite{B}.  These algebras have been
studied before, predominantly in the case when $\vec{p}$ is empty.  In this case, elements of $R_C(1)$ are known as non-abelian theta functions, and the map to projective space on the coarse moduli space defined by $R_C(1)$ is known as the theta-map.  We direct the reader to the article \cite{P} for a survey of what is known about this ring.   For $G = SL_2(\C)$, the ring of non-abelian Theta functions has also recently been shown to be projectively normal by Abe, \cite{A}.  The present paper is in part motivated by an attempt to extend understanding to the parabolic case, when $\vec{p}$ is non-empty.

The algebra $R_{C, \vec{p}}(\vec{\lambda}, L)$ is defined as above only for $(C, \vec{p})$ smooth, however the space 
$V_{C, \vec{p}}(\vec{\lambda}, L)$ makes sense for any stable curve $(C, \vec{p}) \in \bar{\mathcal{M}}_{g, n}.$
 In \cite{M2}, we showed that the direct sum $\bigoplus_{N \geq 0} V_{C, \vec{p}}(N\vec{\lambda}, NL)$ can be given an algebra structure for any stable curve, and over smooth curves this agrees with the algebra structure on $R_{C, \vec{p}}(\vec{\lambda}, L).$
From now on we refer to these fibers $R_{C_{\Gamma}, \vec{p}_{\Gamma}}(\vec{\lambda}, L)$ for a non-smooth curve $(C_{\Gamma}, \vec{p}_{\Gamma})$ with the same notation.  Additionally, these algebras fit together into a flat sheaf of algebras over $\bar{\mathcal{M}}_{g, n}.$
This implies that one can deduce properties of $R_{C, \vec{p}}(\vec{\lambda}, L)$ for generic $(C, \vec{p})$ by studying
the algebra for a particular $(C, \vec{p}) \in \bar{\mathcal{M}}_{g, n}.$  Our strategy is to deduce properties about $R_{C, \vec{p}}(\vec{\lambda}, L)$ by passing to the fiber over non-smooth curves in $\bar{\mathcal{M}}_{g, n},$ where the combinatorics of the factorization rules (see \cite{TUY}, \cite{B}) of conformal blocks can help.

\subsection{Conformal blocks as weighted graphs}

The stack $\bar{\mathcal{M}}_{g, n}$ is stratified by closed substacks indexed by graphs $\Gamma.$  The lowest
components of this stratification are certain closed points indexed by trivalent graphs $\Gamma$ with first Betti number $\beta_1(\Gamma) = g$
and $n$ leaves. The curve $(C_{\Gamma}, \vec{p}_{\Gamma})$ corresponding to a trivalent graph $\Gamma$ is the stable arrangement of marked copies of $\mathbb{P}^1$ with dual graph $\Gamma.$

\begin{figure}[htbp]
\centering
\includegraphics[scale = 0.5]{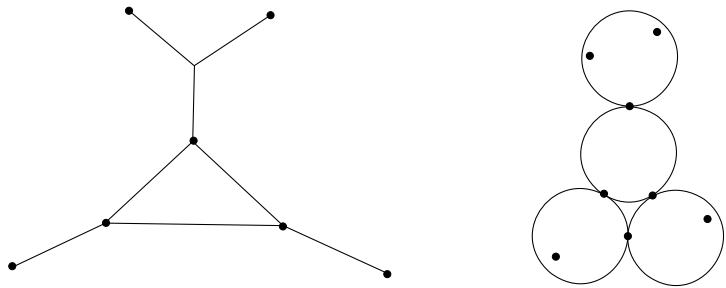}
\caption{The stable curve corresponding to a trivalent graph.}
\label{Fig.5}
\end{figure}

We now restrict our attention to $G = SL_2(\C).$ For this group, dominant weights are non-negative integers $r \in \mathbb{Z}_{\geq 0}.$
In this case, the factorization rules imply that the space of conformal blocks $V_{C_{\Gamma}, \vec{p}_{\Gamma}}(\vec{r}, L)$ over the point corresponding to a graph $\Gamma$ has a distinguished basis given by weightings of the edges of $\Gamma$ by non-negative integers which satisfy a short collection of conditions.  

\begin{definition}\label{polydef}
For $\Gamma$ a trivalent graph with $n$ leaves labeled $1, \ldots, n,$ $\vec{r}$ an $n-$vector of non-negative integers, and $\beta_1(\Gamma) = g,$ we define $P_{\Gamma}(\vec{r}, L) \subset \R^{|Edge(\Gamma)|}$ to be the set of non-negative real weightings of the edges of $\Gamma$ which satisfy the following conditions. 

\begin{enumerate}
\item For any trinode $v \in \Gamma$ with the three incident edge weights $w_1, w_2, w_3$
must satisfy the triangle inequalities: $|w_1 - w_2| \leq w_3 \leq w_1 + w_2.$\\
\item For any trinode as above, $w_1 + w_2 + w_3 \leq 2L.$\\
\item The weight on the edge attached to the $i-$th leaf equals $r_i.$\\
\end{enumerate}
\end{definition}

\begin{figure}[htbp]
\centering
\includegraphics[scale = 0.5]{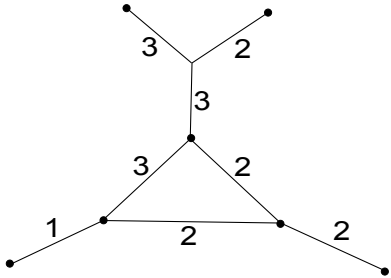}
\caption{A weighting representing a conformal block on a $4-$marked curve of genus $1.$  
The external weights are $1, 3, 2, 2,$ and the level could be anything bigger than $4,$ so this may represent
an element in $V_{C, p_1, p_2, p_3, p_4}(1,3, 2, 2, 4).$}
\label{Fig0}
\end{figure}

We consider this polytope with respect to the lattice of integer points in $\R^{|Edge(\Gamma)|}$ defined by the condition
that the sum $w_1 + w_2 + w_3 \in 2\Z$ for any trinode $v \in \Gamma.$  In addition to indexing a basis of the spaces $V_{C, \vec{p}}(\vec{r}, L),$ the polytope $P_{\Gamma}(\vec{r}, L)$ captures more information about the algebras $R_{C, \vec{p}}(\vec{r}, L),$ in the sense that the edge-wise addition operation on weightings of $\Gamma$ is almost the same as the multiplication operation in these algebras. The following is proved in \cite{StXu} for $C$ genus $0,$ and for general genus by Proposition $1.5$ of \cite{M2}.

\begin{theorem}
Let $\Gamma$ be a trivalent graph with $n$ labeled leaves and first Betti number equal to $g.$
For any curve $(C, \vec{p}) \in \bar{\mathcal{M}}_{g, n}$ the algebra $R_{C, \vec{p}}(\vec{r}, L)$ can be flatly degenerated to the graded semigroup algebra associated to $P_{\Gamma}(\vec{r}, L).$ 
\end{theorem}

\begin{proof}
 Both steps, establishing the flat sheaf of algebras over $\bar{\mathcal{M}}_{g, n},$ and defining and analyzing the ordering,
appear in \cite{M2},( see also \cite{A}).   For any curve $(C, \vec{p}),$ the algebra $R_{C, \vec{p}}(\vec{r}, L)$ is in a flat family with the algebra
over the curve $(C_{\Gamma}, \vec{p}_{\Gamma})$ of type $\Gamma$ by Proposition $1.5$ of \cite{M2}.  By analyzing $R_{C_{\Gamma}, \vec{p}_{\Gamma}}(\vec{r}, L)$ with respect to the distinguished basis given by the factorization rules, it can be found that multiplication in this algebra is "lower triangular" with respect to a natural ordering on the weightings $w$ of $\Gamma.$   That is, if $[w_1]$ and $[w_2]$ are the elements of $R_{C_{\Gamma}, \vec{p}_{\Gamma}}(\vec{r}, L)$ corresponding to weightings $w_1, w_2$
then $[w_1]\times [w_2] = [w_1 + w_2] + $ $lower$ $terms.$  Taking the associated graded algebra of  $R_{C_{\Gamma}, \vec{p}_{\Gamma}}(\vec{r}, L)$  with respect to this term order then yields the graded toric algebra $\C[P_{\Gamma}(\vec{r}, L)].$  A standard Reese algebra construction then
yields a flat family over $\C$ with general fiber $R_{C_{\Gamma}, \vec{p}_{\Gamma}}(\vec{r}, L)$ and special fiber $\C[P_{\Gamma}(\vec{r}, L)].$ 
\end{proof}

The conditions satisfied by the weights $w \in P_{\Gamma}(\vec{r}, L)$ around each internal vertex are called
the quantum Clebsch-Gordon conditions.  Weightings of graphs which satisfy these conditions are known by various names in mathematical physics, 
such as Feynman diagrams or spin diagrams, see \cite{Ko}.  Analysis of these polytopes and other convex sets of spin diagrams also appears
in computational biology, specifically in the work of Buczynska and Wiesniewski on phylogenetic algebraic varieties, \cite{BW}, \cite{Bu}.
The polytope $P_{\Gamma}(\vec{r}, L)$ is a cross-section of the cone $\tau(\Gamma)$ corresponding to Buczynska's graphical phylogenetic toric variety introduced in \cite{Bu} to study the Jukes-Cantor statistical model on graphs.  

A good deal of information is preserved by flat degeneration, such as the degree, the Hilbert function of $R_{C, \vec{p}}(\vec{r}, L),$ and commutative algebra features like the Gorenstein property.  In theory, all of these details can be computed from the polytopes $P_{\Gamma}(\vec{r}, L).$  A non-trivial consequence of this observation is that the equivalent polyhedral information for $P_{\Gamma}(\vec{r}, L),$ the volume and the number of lattice points, etc, is independent of the graph $\Gamma.$     It should be noted that the degeneration technique we are using here applies to other groups $G$ as well, except the resulting degenerations are not toric.  Perhaps this issue can be resolved with a better understanding of the underlying combinatorial representation theory.


\subsection{Statement of results}

Depending on the property of $R_{C, \vec{p}}(\vec{r}, L)$ one wishes to study, certain toric degenerations can be more suited to the task than others.
Next we describe the conditions we place on $\Gamma$ with respect to $\vec{r}$ in order to ensure the polytope $P_{\Gamma}(\vec{r}, L)$ is suited to our needs.
We will use three special graph topologies, depicted in Figure \ref{Fig4}. 

\begin{definition}
A trivalent tree $\tree$ is said to be $caterpillar$ if every vertex is connected by an edge to some leaf. 
\end{definition}

Notice that a caterpillar tree has two pairs of leaves which share a common vertex.  We call 
these paired leaves, and we say they are at the head or tail of the caterpillar. 

\begin{definition}
A trivalent graph $\Gamma$ is said to be a caterpillar graph if it is obtainable from a caterpillar tree
by adding a loop on one of the leaves at the head or tail, or by adding a doubled edge at the midpoint
of one of the edges.    
\end{definition}

\begin{definition}
A trivalent graph $\Gamma$ of genus $g$ with $n$ leaves is $tree-like$ if it
is obtained from a trivalent tree with $g + n$ leaves by adding loops on 
$g$ leaves.  
\end{definition}

\begin{figure}[htbp]
\centering
\includegraphics[scale = 0.4]{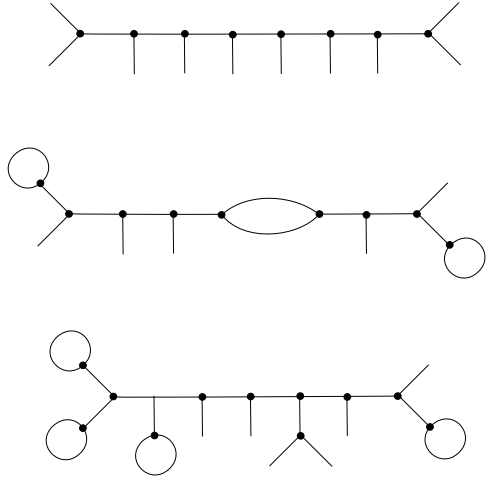}
\caption{From the top, a caterpillar tree, a caterpillar graph, and a tree-like graph.}
\label{Fig4}
\end{figure}

In our analysis of $P_{\Gamma}(\vec{r}, L)$ for $\Gamma$ tree-like, we require that $\vec{r}$ and $\Gamma$ satisfy a condition, which we call compatibility. 

\begin{definition}
For the polytope $P_{\Gamma}(\vec{r}, L)$ we say a vector $\vec{r}$ is "compatible" with a graph $\Gamma$ if any odd $r_i$ is assigned to a leaf-edge which shares a vertex with with the leaf-edge of another odd $r_j.$  
\end{definition}

\begin{figure}[htbp]
\centering
\includegraphics[scale = 0.35]{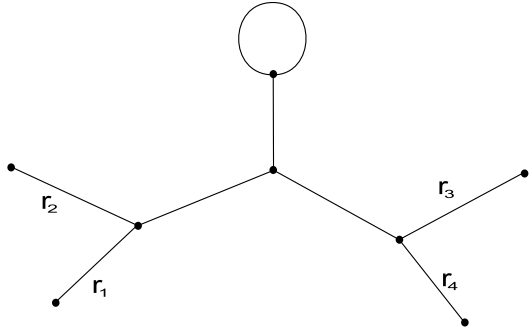}
\caption{$(r_1, r_2, r_3, r_4)$ is compatible with the graph when $r_1 + r_2, r_3 + r_4 \in 2\Z$}
\label{Fig3}
\end{figure}

Recall that for any lattice point in $P_{\Gamma}(\vec{r}, L)$, the weights at any trinode of $\Gamma$ sum to an even number.
It is a simple combinatorial exercise to show that this implies that the number of odd-weighted leaf-edges of $\Gamma$ is always
even.  As a consequence, we obtain that there are only conformal blocks with weights $\vec{r}$ if an even number of the $r_i$ are odd, 
this implies the following.  

\begin{proposition}
Whenever $V_{C, \vec{p}}(\vec{r}, L)$ has dimension $> 0,$ there is a tree-like graph $\Gamma$ which is compatible with $\vec{r}.$
\end{proposition}

The compatibility property was critical in the proof of the following theorem, from \cite{M1}.

\begin{theorem}\label{g1}
Let $\tree$ be a trivalent tree, and $L > 1$, and let $\vec{r}$ be compatible with $\tree,$ then $\C[P_{\tree}(\vec{r}, 2L)]$ is generated in degree $1,$ and the binomial ideal of relations $I_{P_{\tree}(\vec{r}, 2L)}$ is generated by polynomials of degree $\leq 3.$
\end{theorem}
 
\noindent
Our first result is a generalization of Theorem \ref{g1} to tree-like graphs. 

\begin{theorem}\label{polypres}
Let $\Gamma$ be tree-like, and $L >1,$ and $\vec{r}$ compatible with
$\Gamma,$ then $\C[P_{\Gamma}(\vec{r}, 2L)]$ is generated in degree $1,$ and
the binomial ideal of relations $I_{P_{\Gamma}(\vec{r}, 2L)}$ is generated by polynomials of degree $\leq 3.$
\end{theorem}

In \cite{M1} there are examples which show that certain degree $3$ polynomials are necessary
to generate $I_{P_{\Gamma}(\vec{r}, L)}.$  However, when we choose $\Gamma$ to have the caterpillar graph
topology, the relations can become more tractable.

\begin{theorem}\label{polyquad}
Let $\Gamma$ be a caterpillar graph, then $\C[P_{\Gamma}(2\vec{r}, 2L)]$ is generated in degree $1,$ and 
the presenting ideal $I_{P_{\Gamma}(2\vec{r}, 2L)}$ has a quadratic, square-free Gr\"obner basis.
\end{theorem}

A consequence is that any non-empty $P_{\Gamma}(\vec{r}, L)$ becomes "nice" (normal, with quadratically generated binomial ideal) when we take its Minkowski square, $P_{\Gamma}(2\vec{r}, 2L).$ 
Using the flat degeneration, these theorems can be lifted back to the algebras $R_{\C, \vec{p}}(\vec{r}, L).$

\begin{theorem}
For any $\vec{r}$, $L > 1$, and $C$ a generic curve, $R_{C, \vec{p}}(\vec{r}, 2L)$ is generated in degree $1$ with ideal of relations generated in degree $3.$  
For any $\vec{r}, L$, $R_{C, \vec{p}}(2\vec{r}, 2L)$ is Koszul. In particular, it has quadratic relations.  
\end{theorem}

\begin{proof}
We give a sketch of the argument. Let $\Gamma$ and $\vec{r}$ be compatible.  We assume that both properties have been proved for the appropriate toric algebra $\C[P_{\Gamma}(\vec{r}, L)].$  This implies the corresponding statements for the fiber $R_{C_{\Gamma}, \vec{p}_{\Gamma}}(\vec{r}, L)$ by standard properties of associated graded algebras.   The moduli $\bar{\mathcal{M}}_{g, n}$ is a connected, Noetherian Deligne-Mumford stack over $\C,$ therefore for any closed point $q: Spec(\C) \to \bar{\mathcal{M}}_{g, n}$ there is a dense open substack $U$ with an \'etale Noetherian affine cover $Q: Spec(A) \to U.$  The fiber of our sheaf of algebras is then equal to the fiber of the pullback sheaf over $q': Spec(\C) \to Spec(A)$ for some closed point of $Spec(A).$ 

Over a Noetherian affine base,  we obtain the statement on generators and relations from an application of Nakayama's lemma to the graded components of the algebra and its presenting ideal. The statement on the Koszul property results from the fact that Betti numbers of the corresponding generalized Koszul complex (see \cite{E}, Ex $17.22$) do not increase under specialization because their graded components are locally free modules.   We may reduce to this case because the generic point of $\bar{\mathcal{M}}_{g, n}$ is \'etale-covered by the generic point of $Spec(A),$ and the fiber over this point is obtained by base-change.

\end{proof}

\begin{remark}
Naturally, this theorem applies in the case when $\vec{r}$ is empty.  In this case it implies that the above generation and relation properties hold for the projective coordinate ring of the square of the line bundle $\mathcal{L}$ on $\mathcal{M}_{C}(SL_2(\C)),$ for generic $C.$  The generation result is weaker than the result of Abe, \cite{A}, however the result on relations is new.  Abe's strategy is similar to ours, he employs a filtration built from the factorization rules, however he stops short of a toric degeneration.  It is at present unknown to us if Abe's results can be replicated with an inspired choice of graph $\Gamma.$   
\end{remark}

\begin{remark}
When $\Gamma$ is a tree and $L$ is very large, the algebra $\C[P_{\Gamma}(\vec{r}, L)]$ is also realized as a toric degeneration of a projective
coordinate ring of the moduli of $\vec{r}$-weighted points on the projective line, $\mathcal{M}_{\vec{r}} = SL_2(\C) {}_{\vec{r}} \ql [\mathbb{P}^1]^n.$
Theorem \ref{polyquad} then shows that this projective coordinate ring has a quadratic, square-free Gr\"obner basis when $\vec{r}$ is even, a result 
also obtained by Herring and Howard, \cite{HH}.
\end{remark}

\subsection{Building blocks of $P_{\Gamma}(\vec{r}, L)$}

Theorems  \ref{g1} ,\ref{polypres}, and \ref{polyquad} are proved by understanding the structure of $P_{\Gamma}(\vec{r}, L)$ as a fiber product of rational polytopes. In \cite{M1}, the polytopes $P_{\tree}(\vec{r}, L)$ for a trivalent tree $\tree$ were shown to be "balanced," a geometric feature which endows a polytope with several nice algebraic properties. The technical observation that makes these theorems work is the fact that the balanced property is stable with respect to certain kinds of fiber products.

\begin{definition}
Let $P_1 \subset \R^n$, $P_2 \subset \R^m$ and $Q \subset \R^k$ be lattice polytopes, 
and let $\pi_1: \R^n \to \R^k$ and $\pi_2: \R^m \to \R^k$ be lattice maps which take
$P_i$ into $Q.$  The fiber product $P_1 \times_Q P_2 \subset \R^{n + m - k}$ is defined
to be the subset of points $(x, y) \in P\times Q \subset \R^{n+m}$ such that $\pi_1(x) = \pi_2(y).$
\end{definition}

We will prove Theorem \ref{polypres} by following a similar strategy to the program used in \cite{M1}.  The conditions which define $P_{\Gamma}(\vec{r}, L)$ are all localized around individual trinodes of $\Gamma,$ which implies we may represent $P_{\Gamma}(\vec{r}, L)$ as a fiber product of polytopes defined on the components of an "exploded graph" $\hat{\Gamma}.$  This is a graph obtained from $\Gamma$ by splitting any edge in $\Gamma$ which separates the graph into two disjoint components. 

\begin{figure}[htbp]
\centering
\includegraphics[scale = 0.5]{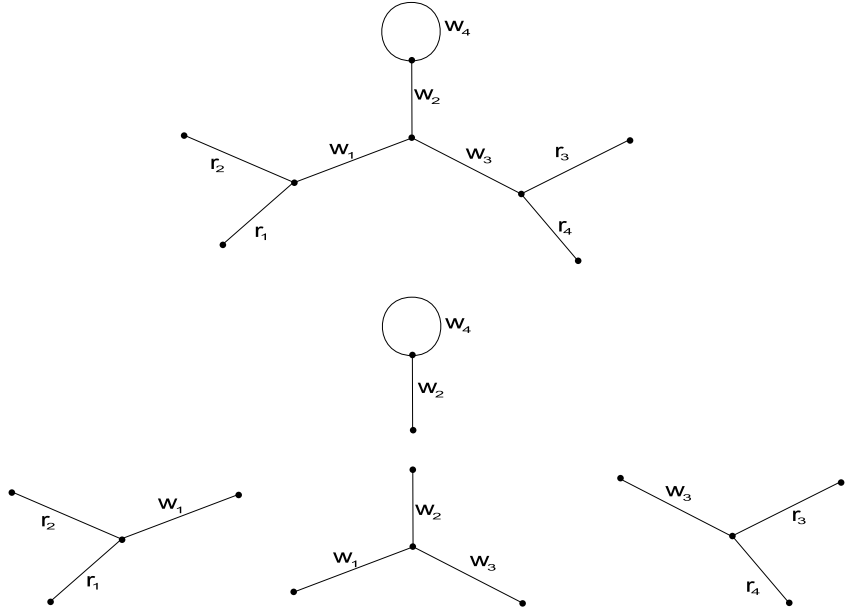}
\caption{Splitting the separating edges of $\Gamma$ (top) produces the union of graphs $\hat{\Gamma}$ (bottom).}
\label{Fig1}
\end{figure}

See also Figure \ref{assemble}.  The connected components of $\hat{\Gamma}$ each have a special polytope associated to them, we call these the "building blocks" of $P_{\Gamma}(\vec{r}, L).$
The strategy is to first analyze the building block polytopes, establishing that they have the algebraic features we want, then show these features are stable under fiber product.   When $\Gamma$ is tree-like, the polytopes given by weightings of the components of $\hat{\Gamma}$ are single loops with an edge, or a trinode with $0$, $1$, or $2$ edges with edge weight fixed to a specific value (say $r,$ or $r,s$).  When $\Gamma$ is caterpillar, no trinodes with $0$ fixed edges appear, but a loop with $2$ edges can appear, see  Figure \ref{blocks} below.  There is a depiction of each polytope beneath its corresponding graph. See Figures \ref{Fig5}, \ref{Fig6}, \ref{Fig6.5}, and \ref{Fig7} for more detailed illustrations of these polytopes.  

\begin{figure}[htbp]
\centering
\includegraphics[scale = 0.55]{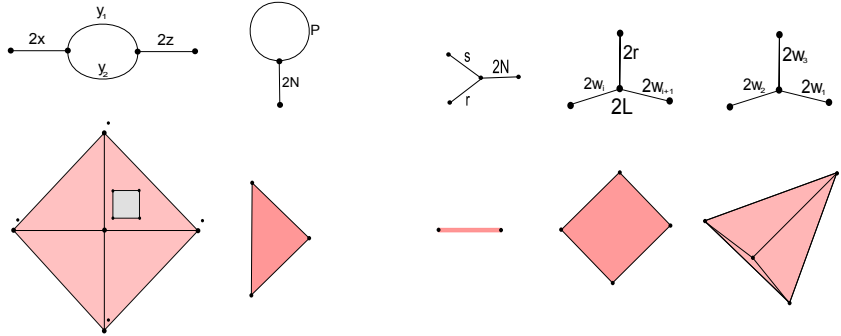}
\caption{The building blocks of $P_{\Gamma}(\vec{r}, L).$}
\label{blocks}
\end{figure}

\noindent
We call the polytope corresponding to a loop with an edge $B(L),$ weightings of a trinode with one or two fixed edges (with value equal to $r, s$) $P_3(r, L)$ or $P_3(r, s, L)$ respectively, general weightings of a trinode $P_3(L),$ and a loop with two edges $B_2(L).$   The polytope $B_2(L)$ is $4$ dimensional, so the image depicted in the figure is a projection into $\R^2.$

Here is where compatibility and the assumption that we work with even level $2L$ are used.  It is another simple combinatorial exercise to verify that when $\Gamma, \vec{r}$ are compatible, all non-leaf and non-loop edges of $\Gamma$ are weighted with an even number.  This greatly simplifies the structure of the building block polytopes.  The main technical part of this paper is to develope and clarify the polytope properties that make this construction work.   As fiber products are categorical operations, we view the natural habitat for the above theorems as closure properties of a category of polytopes under certain fiber products.  We define this category next.

\subsection{The category $\mathcal{P}$}

  For the following definitions see \cite{St}. For a lattice polytope $P \subset \R^n$ let $\Sigma: X_P \to \R$ be an assignment of real numbers to the lattice points of $P.$  The function $\Sigma$ defines a term order on the polynomial ring $\C[X_P].$   For a monomial $x^{\bold{m}} = x_1^{m_1}\ldots x_k^{m_k},$ we define $\Sigma(\bold{m}) = \sum_{i = 1}^k m_i\Sigma(x_i).$  For two monomials, we say $x^{\bold{m}} \leq x^{\bold{n}}$ when $deg(\bold{m}) < deg(\bold{n})$ or $deg(\bold{m}) = deg(\bold{n})$ and  $\Sigma(\bold{m}) \leq \Sigma(\bold{n}).$

The polytope $P$ defines a graded algebra $\C[P],$ where the $m-$th graded component has a basis given by the lattice points $X_{mP}$ of the $m$-th minkowski sum $mP = P + \ldots + P,$ and the product operation is lattice point addition.   Let $I_P \subset \C[X_P]$ be the toric ideal which is the kernel of the map $\C[X_P] \to \C[P].$  For a polynomal $f \in I_P,$ the initial form $in_{\Sigma}(f)$ with respect to $\Sigma$ is defined to be the sum of the monomials in $f$ which are largest under the term order defined by $\Sigma.$ 

\begin{definition}
A monomial $x^{\bold{m}}$ is said to be standard with respect to a term order $\Sigma$ if it is $not$ an initial form of any polynomial in $I_P.$
\end{definition}

For a lattice point $b \in mP,$ we let $\Delta_b$ be the set of monomials
in $\C[X_P]$ which map to $b.$ The difference of any two elements in $\Delta_b$
is a member of $I_P.$   The function $\Sigma$ defines a partial ordering
on $\Delta_b,$ and the minimal elements of this partial ordering are the standard monomials. 

\begin{definition}
We define the category $\mathcal{P}$ to have pairs $(P, \Sigma)$ as objects,
where $P$ is a lattice polytope in some $\R^n,$ and $\Sigma$ is a term order on $P$. 
A morphism $\pi:(P, \Sigma) \to (Q, \Gamma)$ is a linear map on $\pi: P \to Q,$ induced by a map on 
lattices, such that the standard monomials of $\Sigma$ map to standard monomials of $\Gamma.$
\end{definition}

The following is the most general result we will prove about the
category $\mathcal{P}.$  We say $(Q, \Sigma)$ has unique standard monomials
if the sets $\Delta_b$ always have a unique minimal element when they are non-empty. 

\begin{proposition}\label{fiber}
If $\pi_1: (P_1, \Sigma_1) \to (Q, \Sigma)$ 
and $\pi_2: (P_2, \Sigma_2) \to (Q, \Sigma)$
are maps in $\mathcal{P},$ and if $(Q, \Sigma)$ has unique standard monomials, then the fiber product object
$(P_1, \Sigma_1)\times_{(Q, \Sigma)}(P_2, \Sigma_2)$ exists in $\mathcal{P}.$
\end{proposition}

The category of lattice polytopes comes with a fiber product, so the content of this proposition is in showing that there is a term order on $P_1 \times_Q P_2$
with maps to $(P_1, \Sigma_1)$, $(P_2, \Sigma_2)$ satisfying the universal 
property of a fiber product.  Next, we formalize properties of balanced polytopes first observed in \cite{M1} by 
showing that the category $\mathcal{P}$ has a distinguished subcategory that is closed under fiber product.

\begin{definition}
A term order $\Sigma$ on a lattice polytope $P$ is said to be "flag" when
a monomial $x^{\bold{m}}$ from $P$ is standard if and only if every degree $2$ divisor of $x^{\bold{m}}$
is standard, and $x^n$ is standard for any $x \in X_P$. 
\end{definition}

\begin{proposition}\label{flagfiber}
If $\pi_1: (P_1, \Sigma_1) \to (Q, \Sigma)$ 
and $\pi_2: (P_2, \Sigma_2) \to (Q, \Sigma)$
are maps in $\mathcal{P},$ and if $(Q, \Sigma)$ has unique standard monomials, then 
$(P_1, \Sigma_1)$ and $(P_2, \Sigma_2)$ flag implies
$(P_1, \Sigma_1)\times_{(Q, \Sigma)}(P_2, \Sigma_2)$ flag. 
\end{proposition}

Balanced polytopes are examples of polytopes with flag term orders.  Geometrically, the term order divides the polytope
into sub-polytopes, and the standard monomials are simply monomials with components all corresponding to lattice points from the same sub-polytope, called a "standard region."  In the case of balanced polytopes, these regions are all lattice sub-polytopes of the unit cube.  Theorem \ref{polypres} above is a consequence of the following propositions. Recall that a rational lattice polytop $P$ is said to be $normal$ when the graded algebra $\C[P]$ is generated in degree $1,$ by the elements corresponding to the lattice points of $P.$

\begin{proposition}\label{flaggen}
Let $(P_1, \Sigma_1)\times_{(Q, \Sigma)}(P_2, \Sigma_2)$ be a fiber product of flag elements
over $(Q, \Sigma)$ with unique standard monomials.  If each $(P_i, \Sigma_i)$ is normal, then 
$(P_1, \Sigma_1)\times_{(Q, \Sigma)}(P_2, \Sigma_2)$ is normal.  
\end{proposition}

For $(P, \Sigma)$ we call a binomial $x^{\bold{n}} - x^{\bold{m}} \in I_P$ a standard relation
if both monomials are composed of lattice points from the same standard region. 

\begin{proposition}\label{flagrel}
Let $(P_1, \Sigma_1)\times_{(Q, \Sigma)}(P_2, \Sigma_2)$ be a fiber product of flag objects
over $(Q, \Sigma)$ with unique standard monomials.  If the standard relations of $(P_i, \Sigma_i)$
are generated by standard relations of degree $k_i,$ then $I_{(P_1, \Sigma_1)\times_{(Q, \Sigma)}(P_2, \Sigma_2)}$
is generated in degree $max\{k_1, k_2\}.$
\end{proposition}

Properties of fiber products over a toric algebra have been
studied before by Sullivant in \cite{Su}.  These theorems can be viewed
as refinements of Theorem 2.8 in \cite{Su}.  In our results, the assumption
that the base of the fiber product have strict unique factorization has been removed
by restricting our study to fiber products of binomial ideals.

\subsection{The concatenation product on term orders}

 The second technical result, used to prove theorem \ref{polyquad}, is a version of Sullivant's 
Theorem 2.9 in \cite{Su}.  Here we also remove a restriction that the base of the fiber product
have unique factorization by only considering binomial ideals.  We define another operation in $\mathcal{P}.$ 

\begin{definition}
Define $(P_1, \Sigma_1) \boxtimes_{(Q, \Sigma)} (P_2, \Sigma_2)$ in $\mathcal{P}$
on the fiber product polytope $P_1\times_Q P_2$ with term order given 
by the following rule on monomials.  For lattice points $(V_1, W_1), (V_2, W_2) \in (P_1, \Sigma_1) \boxtimes_{(Q, \Sigma)} (P_2, \Sigma_2)$ 
we say $(V_1, W_1) < (V_2, W_2)$ if 

\begin{enumerate}
\item $\Sigma_1(V_1) < \Sigma_1(V_2)$ or
\item the $\Sigma_1$ values are equal, and $\Sigma_2(W_1) < \Sigma_2(W_2).$ 
\end{enumerate}

\noindent
For monomials $x^{(V_1, W_1)\ldots(V_k, V_k)}, x^{(U_1, Y_1)\ldots(U_{\ell}, Y_{\ell})}$ we first order by degree, then with the sum weighting $\Sigma_1 \oplus \Sigma_2,$ we then break ties by writing both sets of exponents in decreasing order with the above ordering and we declare  $(V_1, W_1)\ldots (V_k, W_k) > (U_1, Y_1)\ldots(U_{\ell}, Y_{\ell}),$ if $(V_i, W_i) > (U_i, Y_i)$ at the first place they differ.
\end{definition}

\begin{proposition}\label{quadgrob}
 If the presenting ideals of $(P_1, \Sigma_1),$ $(P_2, \Sigma_2),$ and  $(Q, \Sigma)$ all have quadratic, square-free Gr\"obner
bases with respect to their term orders, then so does the presenting ideal of

 $(P_1, \Sigma_1) \boxtimes_{(Q, \Sigma)} (P_2, \Sigma_2).$
\end{proposition}

Using this, we establish Theorem \ref{polyquad} from the fact that the presenting ideals of intervals $[0, L]$ and the building block polytopes have quadratic, square-free Gr\"obner bases. Note that although we do not assume that the base of our fiber products has unique factorization, we do assume this "locally" by requiring the polytope to have a quadratic square-free Gr\"obner basis.

\subsection{Acknowledgements}

We would like to thank Weronika Buczynska, Anton Dochtermann, Alex Engstr\"om, Ben Howard, Raman Sanyal, and Seth Sullivant for
useful conversations related to this project, and David Eisenbud for his helpful explanations of the behavior of graded algebras
in flat families.   We also thank the reviewer, who pointed us to the work of Abe \cite{A}, the result of which was inspiration for 
stronger results than we originally appeared in this paper.


\section{The category $\mathcal{P},$ and flag term orders}

In this section we review some basic properties of term orders on toric
ideals.  We discuss geometric decompositions of the polytope induced by a term order, and the corresponding algebra
We cover some properties of the category $\mathcal{P},$ flag term orders, 
and balanced polytopes.   First, we recall the initial complex $\Delta_{\Sigma}(P)$ of a term order, see chapter $8$ of \cite{St}.

\begin{definition}
Let $P$ be a lattice polytope, and $\Sigma$ a term order on $P.$
Define the simplicial complex $\Delta_{\Sigma}(P)$ on the vertex set
$X_P$ as follows.  A set $S \subset X_P$ defines a face of
$\Delta_{\Sigma}(P)$ if every monomial with support $S$
is standard. 
\end{definition}

If $S \in \Delta_{\Sigma}(P)$ and $F \subset S$ has a monomial $x^{\bold{m}}$ supported
on its entries which is not standard then one can make a non-standard monomial
supported on $S$ by multiplying by the remaining generators in $S \setminus F,$
this establishes that $\Delta_{\Sigma}(P)$ is a simplicial complex. Morphisms in the category $\mathcal{P}$ preserve the information in these simplicial complexes. 

\begin{proposition}
Let $\pi: (P, \Sigma) \to (Q, \Gamma)$ be a morphism in $\mathcal{P}$
then $\pi$ induces a simplicial map $\pi_*: \Delta_{\Sigma}(P) \to \Delta_{\Gamma}(Q).$
\end{proposition}

\begin{proof}
Let $S$ define a face in $\Delta_{\Sigma}(P).$
Any monomial $x^{\bold{m}}$ with support $S$ 
is standard, implying its image $\pi(x)^{\bold{m}}$ is standard with respect
to $\Gamma$ in $Q.$ Since this holds for any monomial, all monomials
with support $\pi_*(S) = \{\pi(x_i) | x_i \in S\}$ must be standard, 
which means $\pi_*(S) \in \Delta_{\Gamma}(Q).$
\end{proof}

For flag term orders $\Sigma,$ we will be concerned with the convex hulls $|S| \subset P$
of the faces of $\Delta_{\Sigma}(P).$  Let $\Sigma$ be a flag term order, and let $b \in P$
be some rational point.  Then $Nb$ is a lattice point of $NP$ for some $N >> 0.$
Let $x^{\bold{m}}$ be a standard monomial which maps to $b.$   Then $b$ is in the convex
hull of the lattice points defined by the $x_i,$ and by the flag property, this collection defines
a face $S \in \Delta_{\Sigma}(P).$  Similarly, if $b \in |S|$ is any rational point, then we have 

\begin{equation}
b = \sum_{x_i \in S} s_ix_i\\
\end{equation}

\noindent
for rational $s_i.$  This implies for some $N >> 0$ that the monomial defined by $\sum_{x_i \in S} (Ns_i)x_i = Nb$
is standard.  If $b \in |S|$ is a lattice point, this implies $S \cup \{b\}$ is also in $\Delta_{\Sigma}(P)$
as $Nb$ is defines a standard monomial, implying $(N+1)b = \sum_{x_i \in S} Ns_i x_i + b$ defines a standard monomial as well. 
The following proposition shows that the interiors of maximal faces of $\Delta_{\Sigma}(P)$ do not intersect.

\begin{proposition}
Let $(P, \Sigma)$ be flag, and let $S, T \subset \Delta_{\Sigma}(P),$ and suppose the intersection
$int|S| \cap int|T|$ is non-empty.  Then $S \cup T$ is also in $\Delta_{\Sigma}(P).$  
\end{proposition}

\begin{proof}
Let $b \in int|S| \cap int|T|,$ be a rational point, this implies that $b$ can be written as convex sums in both sets $S$ and $T.$

\begin{equation}
b = \sum_{x_i \in S} r_ix_i = \sum_{y_j \in T}t_iy_i\\ 
\end{equation} 

\noindent
In particular the numbers $t_j$ and $s_i$ can be taken non-zero for all $i, j.$
Choose $N >>0$ so that $Ns_i$ and $Nt_j$ are integers for all $i, j,$ then 
we have found monomial expressions for $Nb \in NP$ with support in $S$ and $T$
respectively.  This implies that these monomials are both standard, and further implies
that the following expressions correspond to standard monomials. 

\begin{equation}
2Nb = \sum_{x_i \in S} (2Nr_i)x_i = \sum_{x_i \in S} (Nr_i)x_i + \sum_{y_j \in T} (Nt_i)y_i\\  
\end{equation}

\noindent
So the set $S \cup T$ is also in $\Delta_{\Sigma}(P).$
\end{proof}
\noindent
Taken together, these observations imply the following. 

\begin{proposition}\label{subd}
Let $(P, \Sigma)$ be flag.  The polytope $P$ is geometrically subdivided into convex regions
$S_i \subset P,$ the convex hulls of the maximal facets of $\Delta_{\Sigma}(P).$ The lattice
points in a region $S_i$ are precisely the members of the corresponding maximal facet. 
\end{proposition}

For any rational point $b \in S$ a maximal face, we can find some set $\{x_0, \ldots, x_n\} \subset S$ such that $b$ is in the convex hull of $x_0, \ldots, x_n.$  If this set is not equal to $S,$ this can only be because two monomials in some $\Delta_{Nb},$ with $N>>0$ get the same minimial $\Sigma$ weight, so $\Sigma$ cannot be a total term order. Conversely, when $\Sigma$ defines a total term order on monomials, all maximal faces $S \in \Delta_{\Sigma}(P)$ are simplices.  

We will now show that that when $(P, \Sigma)$ is flag, much of the algebra of $\C[P]$ is captured by the algebras $\C[S_i],$ $S_i \in \Delta_{\Sigma}(P).$ 
The following lemma allows us to reduce generation and relation degrees for $\C[P]$ to those of the $\C[S_i].$

\begin{lemma}\label{d2}
For any monomial $a_1\ldots a_N \in P$ where $(P, \Sigma)$ is flag, 
$a_1 \ldots a_N$ is related to a 
standard monomial by degree $2$ relations.
\end{lemma}

\begin{proof}
Let $a_1 + \dots a_N = b \in NP,$ and recall $\Delta_b,$ the set of monomials which map to $b.$ 
If $x^{a_1\ldots a_N}$ is standard, there is nothing to do, so suppose this is not the case, then by
definition of flag term orders, it has a non-standard degree $2$ divisor, say $x^{a_1a_2}.$
We replace $x^{a_1a_2}$ with $x^{A_1A_2}$ standard, then 

\begin{equation}
\Sigma(a_1 a_2 \ldots a_N) > \Sigma(A_1A_2\ldots a_N)\\
\end{equation}

\noindent
We repeat this procedure on $A_1A_2\ldots a_N.$  This algorithm must terminate by the
finiteness of $\Delta_b.$
\end{proof}

\begin{corollary}
For $(P, \Sigma)$ flag with maximal facets $S_1, \ldots S_k$ of $K_{\Sigma},$
any relation $x^{a_1\ldots a_n} = x^{b_1\ldots b_N}$ can be taken to a relation in some $S_i$
by degree $2$ relations. 
\end{corollary}

\begin{proof}
We must have $a_1 + \ldots + a_N = b_1 + \ldots + b_M = b$ for some $b \in NP,$ 
therefore any standard factorization of $b$ must be made from the lattice points
of an $S_i$ which contains $\frac{1}{N}b.$  The previous proposition shows that 
both sides of this relation can be converted to such a standard relation by degree $2$
moves.  
\end{proof}
 
The content of this last corollary is that generating sets of relations for the $S_i,$
along with the degree $2$ relations of the form $x^{a_1 a_2} = x^{A_1 A_2},$ with one side
standard, suffice to generate the binomial ideal $I_P.$

\subsection{Balanced polytopes}

We now define a special term order function $\Sigma^2.$  We will recall the definition of balanced polytopes from \cite{M1}, and show that 
they always have $\Sigma^2$ as a flag term order. 

\begin{definition}
Define $\Sigma^2: \mathbb{Z}^n \to \Z$ to be the function which takes
a vector $(a_1, \ldots, a_n)$ to the sum of the squares of its entries $\sum_{i = 1}^n a_i^2.$
\end{definition}

\begin{example}
To motivate the use of this function, consider a number $b \in \Z_{\geq 0}$ and the set $\Delta_N(b)$ of $N$ tuples $(x_1, \ldots, x_N) \in \Z_{\geq 0}^N$ such that $\sum x_i = b.$  We can partially order these tuples by weighting them with $\Sigma^2.$ For $b = 2$,
the set $\Delta_3(3) = \{ [003], [012], [111]\},$ and these are ordered by $\Sigma^2$ as follows,

\begin{equation}
[111] < [012] < [003].\\
\end{equation}

\end{example}

We call the operation with takes a pair of numbers $(x, y)$ to $(\lfloor \frac{x + y}{2} \rfloor, \lceil \frac{x + y}{2} \rceil)$ "balancing."   
For any pair of numbers $x, y$ we have that $x^2 + y^2 \geq \lfloor \frac{x + y}{2} \rfloor^2 + \lceil \frac{x + y}{2} \rceil^2,$ while 
$x + y = \lfloor \frac{x + y}{2} \rfloor + \lceil \frac{x + y}{2} \rceil.$  For any pair $x_i, x_j$ from the a tuple as above, balancing yields a tuple which still sums to $b,$
with a lower $\Sigma^2$ value. Furthermore, the number $\Sigma x_i^2$ is minimized over the set of $n$ non-negative integers numbers with sum $b$ exactly when any $x_i$ and $x_j$ differ by at most $1.$  This implies that the minimal tuple with respect to $\Sigma^2$ is of the form $(x+1, x+1, \ldots, x, \ldots, x)$ where $b = Nx + M$ with $M$ the number of entries of the form $x + 1.$    We can play this same game with a tuple of vectors  $(\vec{x}_1, \ldots, \vec{x}_N) \subset \Z_{\geq 0}^m$ which sum to a fixed vector $\vec{b} \in \Z^n.$  We call the operation  that takes two vectors $\vec{x}, \vec{y}$ to the vectors with entries the balancing of the pairs of entries of $\vec{x}, \vec{y}$ a $balancing$ of $m$-vectors.  Once again, this operation 
lowers the $\Sigma^2$ value of $(\vec{x}, \vec{y}),$ unless these vectors were already balanced with respect to each other.  

\begin{proposition}\label{balnum}
For any tuple of vectors $\{\vec{x}_1, \ldots, \vec{x}_N\} \subset \Z_{\geq 0}^m$
There is a tuple $\{\vec{b}_1, \ldots, \vec{b}_N\} \subset \Z_{\geq 0}^m$ with the following
properties. 
\begin{enumerate}
\item $\sum \vec{b}_i = \sum \vec{x}_i$\\
\item Each tuple of numbers $\{b_1^i, \ldots, b_M^i\}$ is balanced. 
\item $\Sigma^2(\vec{b}_1, \ldots, \vec{b}_N) \leq \Sigma^2(\vec{x}_1, \ldots, \vec{x}_N)$,\\
with equality only when each tuple of numbers $\{x_1^i, \ldots, x_M^i\}$ is balanced.\\
\item  The $\vec{b}_i$ all lie in a common translate of the unit cube in $\Z_{\geq 0}^m.$
\end{enumerate}
\end{proposition}

\begin{proof}
We construct a $\{\vec{b}_1, \ldots, \vec{b}_N\}$ by doing pairwise balancings, beginning
with the set $\{\vec{x}_1, \ldots, \vec{x}_N\}.$  Since each operation lowers the $\Sigma^2$ value, this process terminates.  Number $1$ above follows because this algorithm does not change the total sum of the vectors, numbers $2$ and $3$ are consequences of the algorithm and imply number $4.$
\end{proof}

\begin{definition}
We say a lattice polytope $P \subset \R^m$ is balanced if each tuple of lattice points
$b_1, \ldots, b_N \in P$ has a balancing in $X_P,$ the lattice points of $P.$ 
\end{definition}

\begin{proposition}
A polytope $P \subset \R^n$ is balanced if and only if $(P, \Sigma^2)$ is flag, and the convex hulls of the
maximal elements in $\Delta_{\Sigma^2}(P)$ are subpolytopes of translates of the unit $n-$cube. 
\end{proposition}

\begin{proof}
If $P$ is balanced, then the balanced monomials of $P$ are clearly
the standard monomials with respect to $\Sigma^2,$ so we verify
that the flag properties are satisfied.  First note that for any $x \in X_P,$ $x^n$ is a balanced monomial.
Furthermore, a monomial $x^{\bold{n}}$ is balanced if and only if the entries in any two generators
dividing $x^{\bold{n}}$ differ by $1$ or $0,$ which is the case if and only if any degree $2$
divisor of $x^{\bold{n}}$ is balanced.  This is exactly the flag condition. 

If $(P, \Sigma^2)$ is flag, and the fundamental regions $S_i$ are all subsets
of translates of the unit cube then the algorithm in lemma \ref{d2} above gives a balancing
of any monomial.  
\end{proof}

\begin{corollary}
Let $P$ be a balanced polytope, then $P$ is normal if and only if each maximal cubical region $C_i$ is normal.  Furthermore, the ideal $I_P$
is generated in degree bounded by the degrees required to generate the ideals $I_{C_i}.$
\end{corollary}

\section{Fiber products of polytopes}

The previous section establishes that balanced polytopes make up a special
subclass of polytopes with a flag term order, which make up a special class
of polytopes with term orders.  In this section we show that each of these classes
is closed under special types of fiber products. The guiding principle here is that fiber products will exist
and be well-behaved when the base has unique factorization properties.  First we present
a useful motivating proposition. 

\begin{proposition}
Let $P_1$ and $P_2$ be normal polytopes with maps $\pi_1$, $\pi_2$ to $Q$ a normal polytope with unique factorization,
then $P_1 \times_Q P_2$ is normal.
\end{proposition}  

\begin{proof}
Let $(b_1, b_2) \in N[P_1 \times_Q P_2],$ with $b = \pi_1(b_1) = \pi_2(b_2) \in NQ.$ By the normality of $P_1$ and $P_2,$
we have $b_1 = x^{\bold{n}}$ and $b_2 = y^{\bold{m}}$ for monomials in the lattice points of $P_1$ and $P_2$ respectively. 
We must also have $\pi_1(x^{\bold{n}}) = \pi_2(y^{\bold{m}}),$ so the the components of these monomials must be the same up
to reordering, this means for every component $x_i$ we can find a $y_i$ a component of $y^{\bold{m}}$ such that $\pi_1(x_i) = \pi_2(y_i).$
This allows us to express $(b_1, b_2)$ as a product of lattice points of $P_1\times_Q P_2.$
\end{proof}

For two elements $(P_1, \Sigma_1), (P_2, \Sigma_2) \in \mathcal{P}$ the product polytope $P_1 \times P_2$ has a natural term order given
by $\Sigma_1 \oplus \Sigma_2(u, v) = \Sigma_1(u) + \Sigma_2(v).$  This order also makes sense on 
a fiber product $P_1 \times_Q P_2 \subset P_1 \times P_2.$  We now prove that the fiber product of $(P_1, \Sigma_1), (P_2, \Sigma_2) \in \mathcal{P}$ over an element $(Q, \Sigma)$ with unique standard monomials is a fiber product object in $\mathcal{P}.$

\begin{proof}[Proof of proposition \ref{fiber}]
We show that a monomial $x^{[a_1, b_1] \ldots [a_k, b_k]}$ is standard with respect
to $\Sigma_1 \oplus \Sigma_2$ if and only if $x^{a_1, \ldots, a_k}$ and $y^{b_1, \ldots, b_k}$
are standard with respect to $\Sigma_1$ and $\Sigma_2$ respectively.  This establishes
that the projection maps to $P_1$ and $P_2$ are morphisms in $\mathcal{P},$ and
that any polytope $(D, \Gamma) \in \mathcal{P}$ with a map to $(Q, \Sigma)$ which factors by $\pi_1$ and $\pi_2$ must have a map to the fiber product. 


If both $x^{a_1, \ldots, a_k}$ and $y^{b_1, \ldots, b_k}$ are standard with respect to $\Sigma_1$
and $\Sigma_2,$ then $x^{[a_1, b_1], \ldots, [a_k, b_k]}$ must be minimal with respect to $\Sigma_1 \oplus \Sigma_2,$ as any alternative factorization induces alternative factorizations of $a_1 +\ldots + a_k$ and $b_1 + \ldots + b_k.$  So we must show that $x^{[a_1 b_1] \ldots [a_k b_k]}$
standard implies that $x^{a_1, \ldots a_k}$ and $y^{b_1 \ldots b_k}$ are standard. 

Suppose  $x^{[a_1, b_1], \ldots, [a_k, b_k]}$ is a monomial in $P_1\times_Q P_2$ with
$x^{a_1 \ldots a_k}$ not standard, we will show that this implies that $x^{[a_1, b_1], \ldots, [a_k, b_k]}$ is not standard. 
Let $x^{A_1 \ldots A_k}$ be a standard monomial for $a_1 + \ldots + a_k,$ and
$y^{B_1 \ldots B_k}$ the same for $b_1 + \ldots + b_k.$  Then $\pi_1(A_1 \ldots A_k) = \pi_2(B_1 \ldots B_k) = \pi_1(a_1 \ldots a_k) = \pi_2(b_1 \ldots b_k)$ and it can be arranged that $\pi_1(A_i) = \pi_2(B_i)$ by the unique standard monomial property of $Q.$  This allows us to
form the fiber product monomial $x^{[A_1 B_1] \ldots [A_k B_k]},$ which must be a member of 
$\Delta_{[a_1 b_1] + \ldots + [a_k b_k]}.$   However, $\Sigma_1 \oplus \Sigma_2([A_1 B_1] \ldots [A_k B_k] = \Sigma_1(A_1 \ldots A_k) + \Sigma_2(B_1 \ldots B_k) < \Sigma_1(a_1 \ldots a_k) + \Sigma_2(b_1 \ldots b_k) = \Sigma_1 \oplus \Sigma_2([a_1 b_1] \ldots [a_k b_k]).$ 
\end{proof}

\noindent
As a corollary we obtain a proof of proposition \ref{flagfiber}.

\begin{proof}[Proof of proposition \ref{flagfiber}]
Let $x^{[a_1, b_1] \ldots [a_k, b_k]}$ be a monomial for $P_1\times_Q P_2$ such that
each pair $x^{[a_i, b_i][a_j, b_j]}$ is standard.  By the previous proposition this implies
that both $a_1, \ldots, a_k$ and $b_1, \ldots, b_k$ define standard monomials, which implies
their image is standard as a $Q$ monomial, and therefore $x^{[a_1, b_1] \ldots [a_k, b_k]}$ is standard. 
\end{proof}

\noindent
We also get an expression for the simplicial complex $\Delta_{\Sigma_1 \oplus \Sigma_2}(P_1 \times_Q P_2).$


\begin{proposition}\label{flagprod}
Let $(P_1, \Sigma_1)$ and $(P_2, \Sigma_2)$ be flag, then a maximal facet $S \in \Delta_{\Sigma_1 \oplus \Sigma_2}(P_1 \times_Q P_2).$ is of the form $S_1 \times_Q S_2$ for $S_1 \in \Delta_{\Sigma_1}(P_1)$ and $S_2 \in \Delta_{\Sigma_2}(P_2).$
\end{proposition}

\begin{proof}
Since any standard support is a subset $S \subset S_1 \times_Q S_2,$ it suffices to show that $S_1 \times_Q S_2$ is a facet of $\Delta_{\Sigma_1 \oplus \Sigma_2}(P_1 \times_Q P_2).$ 
Let $[a_1, b_1] \ldots [a_k, b_k]$ be the list of all points in $S_1 \times_Q S_2.$  Any monomial $x^{\bold{m}}$ with support in this set
maps to monomials $\pi_1(\bold{m}), \pi_2(\bold{m})$ with support in $S_1$ and $S_2$ respectively, this means that $\pi_1(\bold{m})$
and $\pi_2(\bold{m})$ define standard monomials, so $x^{\bold{m}}$ must be standard as well.  
\end{proof}

When $(P_1, \Sigma_1),$ $(P_2, \Sigma_2),$ and $(Q, \Sigma)$ are flag, a facet $S \in \Delta_{\Sigma_1 \oplus \Sigma_2}(P_1 \times_Q P_2).$
corresponds to a convex region in $P_1\times_Q P_2.$ By proposition \ref{flagprod} above, $S$ is the fiber product 
polytope of its images in $P_1$ and $P_2$ over a maximal facet of $\Delta_{\Sigma}(Q),$ a
sub-polytope of $Q$ with unique factorization.  When we restrict our attention to balanced polytopes, a fiber product $(P_1, \Sigma^2) \times_{(Q, \Sigma^2)} (P_2, \Sigma^2)$
where the maps are coordinate projections, gives a flag pair $(P_1 \times_Q P_2, \Sigma^2 \oplus \Sigma^2).$  By our discussion, the standard regions
of this term order are all fiber products of cubical subpolytopes of $P_1$ and $P_2$ over cubical subpolytopes
of $Q,$ so they are all cubical.  However, the term order $\Sigma^2 \oplus \Sigma^2$ is double $\Sigma^2$ over the coordinates from $Q.$

\begin{equation}
\Sigma^2 \oplus \Sigma^2(V, x, W) = \Sigma^2(V) + 2\Sigma^2(x) + \Sigma^2(W)\\
\end{equation}

\noindent
So this is not quite the $\Sigma^2$ term order.  An equivalence in the category $\mathcal{P}$ occurs when 
a polytope $P$ has two term orders with the same standard monomials, we will show that this occurs when 
we take the fiber product of balanced polytopes. 

\begin{proposition}
Let $P_1 \times_Q P_2$ be a fiber product of balanced polytopes with $\pi_1: P_1 \to Q$ and $\pi_2:P_2 \to Q$
induced by coordinate projections.  Then the product term order $\Sigma^2 \oplus \Sigma^2$ has the same standard monomials
as the term order given by $\Sigma^2$ on the fiber product.
\end{proposition}

\begin{proof}
A monomial $x^{\bold{n}}$ is standard  with respect  to $\Sigma^2 \oplus \Sigma^2$ if and only if it is of the form $x^{(V_1, x_1, W_1) \ldots (V_k, x_k, W_k)}$
where $x^{(V_1, x_1) \ldots (V_k, x_k)}$ and $y^{(x_1, W_1) \ldots (x_k, W_k)}$ are standard for $(P_1, \Sigma^2)$ and $(P_2, \Sigma^2),$ and 
are therefore balanced.  This is the case if and only if the tuple $(V_1, x_1, W_1) \ldots (V_k, x_k, W_k)$ is balanced. 
\end{proof}

In particular, a fiber product of balanced polytopes over maps which are coordinate projections is a balanced polytope. 
We are interested in fiber products primarily because we can control their generators
and relations.  We establish this by proving propositions \ref{flaggen} and \ref{flagrel}.  These will be used to lift the commutative algebra properties from the building blocks discussed in the last section to $P_{\Gamma}(\vec{r}, L).$

\begin{proof}[Proof of Proposition \ref{flaggen}]
The polytope $(P_1 \times_Q P_2, \Sigma_1 \oplus \Sigma_2)$ is flag by proposition \ref{flagfiber},
so it is tiled by the convex hulls of the maximal facets of $\Delta_{\Sigma_1 \oplus \Sigma_2}(P_1\times_1 P_2).$
By proposition \ref{flagprod} these are all fiber products of normal polytopes over polytopes
with unique factorization, and are therefore normal. 
\end{proof}

\begin{proof}[Proof of Proposition  \ref{flagrel}]
If $x^{(a_1, b_1) \ldots (a_k, b_k)} - x^{(A_1, B_1) \ldots (A_k, B_k)}$ is in the ideal for the fiber product
then it can be converted to a relation in some maximal facet of $\Delta_{\Sigma_1 \oplus \Sigma_2}(P_1\times_Q P_2)$ by degree
$2$ relations, so we may assume without loss of generality that $x^{a_1 \ldots a_k} = x^{A_1 \ldots A_k}$
and $y^{b_1 \ldots b_k} = y^{B_1 \ldots B_k}$ are relations among standard monomials in $P_1$ and $P_2$
respectively. Any relation among standard monomials $x^{a_1 \ldots a_k} - x^{a_1'\ldots a_k'}$ can be lifted in some way 
to a relation $x^{(a_1, b_1) \ldots (a_k, b_k)} - x^{(a_1', b_{i_1})\ldots (a_k', b_{i_k})}$ in $P_1\times_Q P_2$ 
because $Q$ has unique factorization.  This way, $(a_1, b_1)\ldots (a_k, b_k)$ can be converted
to $(A_1, B_1)\ldots (A_k, B_k)$ by relations among standard monomials of $P_1$ and $P_2.$
\end{proof}

\section{Fiber products and quadratic square-free Gr\"obner bases}

The class of polytopes with flag term orders has a distinguished subclass given by those total term orders
with a quadratic square-free Gr\"obner basis.  The "quadratic" part of this distinction is equivalent to the condition that any monomial with standard degree $2$ divisors must itself be standard, and "square-free" implies that the powers of any lattice point must be standard.    The standard regions $S_i$ of these polytopes are unit simplices.  We now know that fiber products in this subcategory yield normal polytopes with a flag term order, however there is not enough information in the fiber product order $\Sigma_1 \oplus \Sigma_2$ to yield a total term order, and a Gr\"obner basis. The issue is that for standard monomials defined by $\vec{v}, \vec{x}, \vec{w}$ in $P_1$, $Q$ and $P_2$ respectively, there could be many ways to form a fiber product
monomial, 

\begin{equation}
(v_1, x_1, w_1) \ldots (v_k, x_k, w_k).
\end{equation}

\noindent
If two elements $w_i$ and $w_j \in P_1$ map to the same $x_j \in Q,$ then we could plausibly permute these entries to obtain distinct, new
standard monomials with respect to $\Sigma_1 \oplus \Sigma_2.$    The standard monomials in the fiber product are exactly those obtained by fixing $v_1, \dots, v_n$ in some order and permuting the $w_1, \ldots, w_k$ while respecting the shared $x_1, \ldots, x_k,$ so
we must establish an order on these sets.  Each standard region in $P_1\times_P P_2$ is
a fiber product of simplices  with well-ordered lattice points over a simplex with well-ordered 
lattice points, so to prove proposition \ref{quadgrob}, we can reduce to the following proposition. 

\begin{proposition}
Let $P_1$, $P$ and $P_2$ be unit simplices each with a fixed well-ordering $\Sigma_1, \Sigma, \Sigma_2$ on their lattice points.  Then $P_1 \times_P P_2$ has a quadratic square-free Gr\"obner basis defined by the orderings $\Sigma_1, \Sigma,$ and $\Sigma_2$
\end{proposition}

\begin{proof}
This follows from a modification of Corollary 2.11 in \cite{Su}.
To define the new term-order, we use the $\boxtimes$ term order.  We have $(V_1, W_1) < (V_2, W_2)$ if $V_1 < V_2$ or $V_1 = V_2$ and $W_1 < W_2,$
and

\begin{equation}
[(V_1^1, W_1^1)\cdots (V_k^1, W_k^1)] > [(V_1^2, W_1^2)\cdots (V_k^2, W_k^2)]\\
\end{equation}

\noindent
if $(V_i^1, W_i^1) = (V_i^2, W_i^2)$ for all $i < \ell$ for some $\ell$
and $V_{\ell}^1 > V_{\ell}^2$ or $V_{\ell}^1 = V_{\ell}^2$ with $W_{\ell}^1 > W_{\ell}^2,$ where we've listed the terms of the monomials in decreasing order.  Assume now that both monomials map to the same lattice point, then the sets $\{V_i^1\}$ and $\{V_i^2\}$ are the same, as are $\{W_i^1\}$ and $\{W_i^2\},$ so we may assume without loss of generality 
that $V_i^1 = V_i^2.$  In this case, if $W_{\ell}^1 > W_{\ell}^2$ then $W_m^1 = W_{\ell}^2$ for some $m > \ell,$ so for the first monomial we can form the relation

\begin{equation}
(V_{\ell}, W_{\ell}^1)(V_m, W_m^1) = (V_{\ell}, W_m^1)(V_m, W_{\ell}^1)\\
\end{equation}

\noindent
Performing this exchange must yield a lower monomial.
This implies that an arbitrary monomial can be taken to the associated standard
monomial with degree $2$ weight-lowering relations.   Now consider a monomial of the form
$x^{(V, W), \ldots, (V, W)},$ and let $x^{(V_1, W_1), \ldots, (V_k, W_k)}$ be another standard
monomial with the same $\Sigma_1 \oplus \Sigma_2$ weight.  This implies that $(V_1, \ldots, V_k) = (V, \ldots, V)$
and $(W_1, \ldots, W_k) = (W, \ldots, W),$ since both $\Sigma_1$ and $\Sigma_2$ define total term orders.   
\end{proof}

We remark that the argument above can be adapted to show that if 
$(P_1, \Sigma_1)$ and $(P_2, \Sigma_2)$ are flag, then the so is $P_1 \times_Q P_2$ with respect
to the $\boxtimes$ term order.








\section{The building blocks of $P_{\Gamma}(\vec{r}, L)$}
In this section we construct and study the building block polytopes, showing that these are all balanced polytopes.  Recall
that we always take fiber product over the interval $[0, L].$  This polytope is balanced,
and the maximal facets of $\Delta_{\Sigma^2}([0, L])$ are the unit length intervals, $[k, k+1].$  

The polytope $P_3(L)$ is dimension $3$ and $B_2(L)$ is dimension $4$, but the other building blocks have dimension $\leq 2.$
In the case of a balanced dimension $2$ polytope $P,$ every intersection of a translate of the unit square with $P$ must be a lattice polytope, so
 $P$ is a convex union of squares and triangles. In particular, all the facets of $P$ are parallel to the lines $y = 0, x = 0, y = x$ or $y = - x,$ 
and conversely any lattice polytope in $\R^2$ with this property is balanced. 

\begin{proposition}
Let $P$ be a balanced polytope in $\R^2,$ then $I_P$ has a quadratic, square-free
Gr\"obner basis. 
\end{proposition}

\begin{proof}
We define a term order on $\C[X_P]$  and leave it to the reader to verify that it is
quadratic and square-free.  First order the elements by degree, then by $\Sigma^2,$ 
then by the Lexicographic ordering on the square, where $[1, 1] > [1, 0] > [0, 1] > [0,0].$
\end{proof}

From this result we can deduce a corollary for dimension $2$ balanced polytopes. 

\begin{theorem}\label{d2bp}
Let $P$ be a fiber product of a finite number of dimension $2$ balanced
polytopes over balanced polytopes of dimension $2$ or $1.$  Then $I_P$
has a Quadratic Square-free Gr\"obner Basis. 
\end{theorem}

From now on we use the assumption that the parameters $\vec{r}$ are adapted to the graph in question.  In this case, the condition that every trinode $v \in \Gamma$ must have an even sum implies that every non-loop and non-leaf edge must have an even weight.  For a tree-like graph $\Gamma,$ the disconnected
graph $\hat{\Gamma}$ has four kinds of components, a single loop with a pendant edge, and trinodes with $2, 1,$ or $0$ pendant leaf edges. For a caterpillar graph $\hat{\Gamma}$ the components
are all  loops with pendant edges, or trinodes with $1$ or $2$ fixed edges.  For a tree-like graph, the building blocks we must consider are exactly those pictured in Figures \ref{Fig5}, \ref{Fig6}, \ref{Fig6.5}, and \ref{Fig7}.  For a caterpillar graph we also must use the polytope pictured in Figure \ref{block2}.

\begin{figure}[htbp]
\centering
\includegraphics[scale = 0.34]{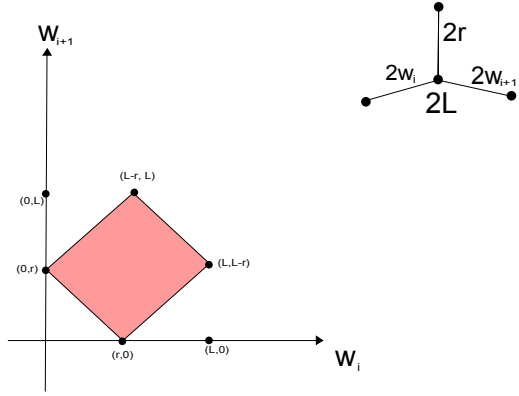}
\caption{The polytope $P_3(r, L)$}
\label{Fig6}
\end{figure}

\begin{figure}[htbp]
\centering
\includegraphics[scale = 0.34]{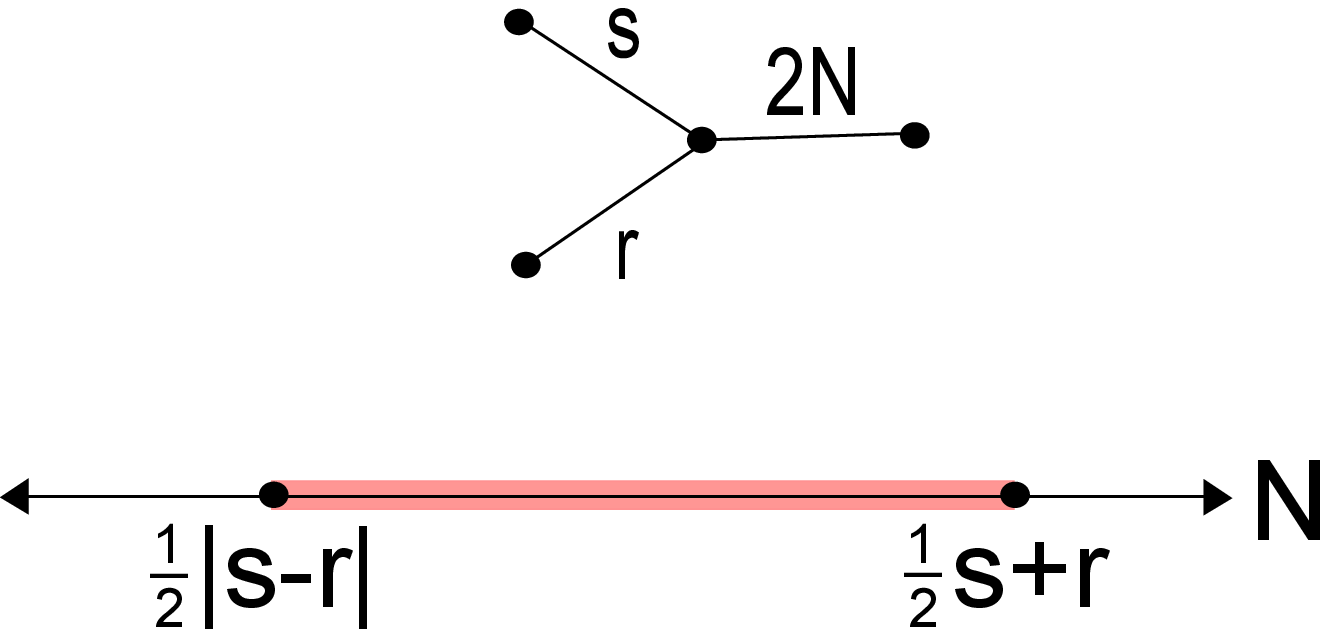}
\caption{The polytope $P_3(r,s,L)$}
\label{Fig6.5}
\end{figure}

\begin{figure}[htbp]
\centering
\includegraphics[scale = 0.34]{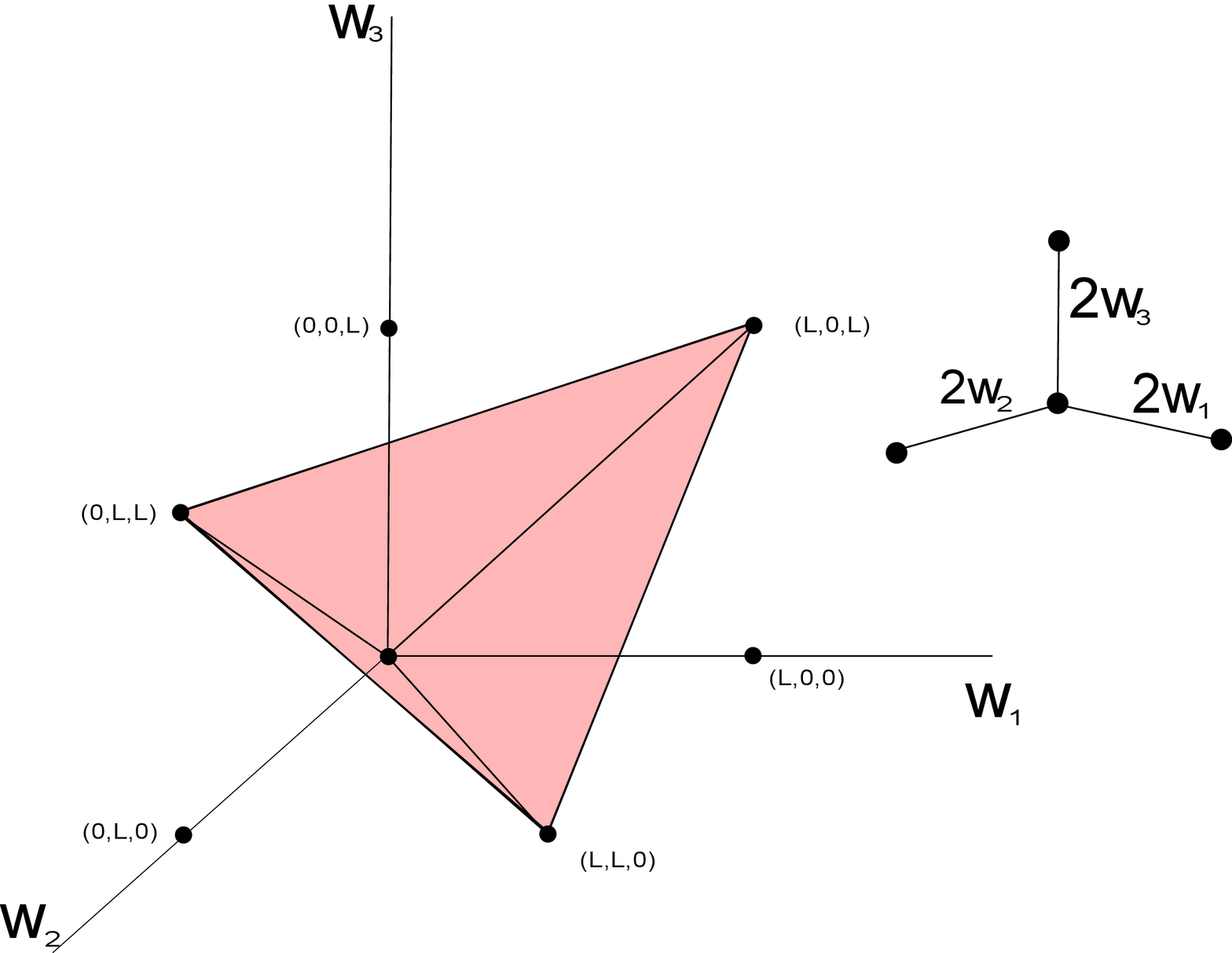}
\caption{The polytope $P_3(L).$}
\label{Fig7}
\end{figure}

\begin{figure}[htbp]
\centering
\includegraphics[scale = 0.34]{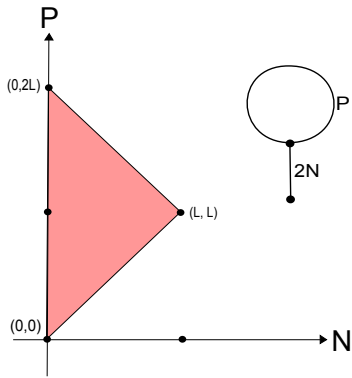}
\caption{The polytope $B(L).$}
\label{Fig5}
\end{figure}

\begin{figure}[htbp]
\centering
\includegraphics[scale = 0.45]{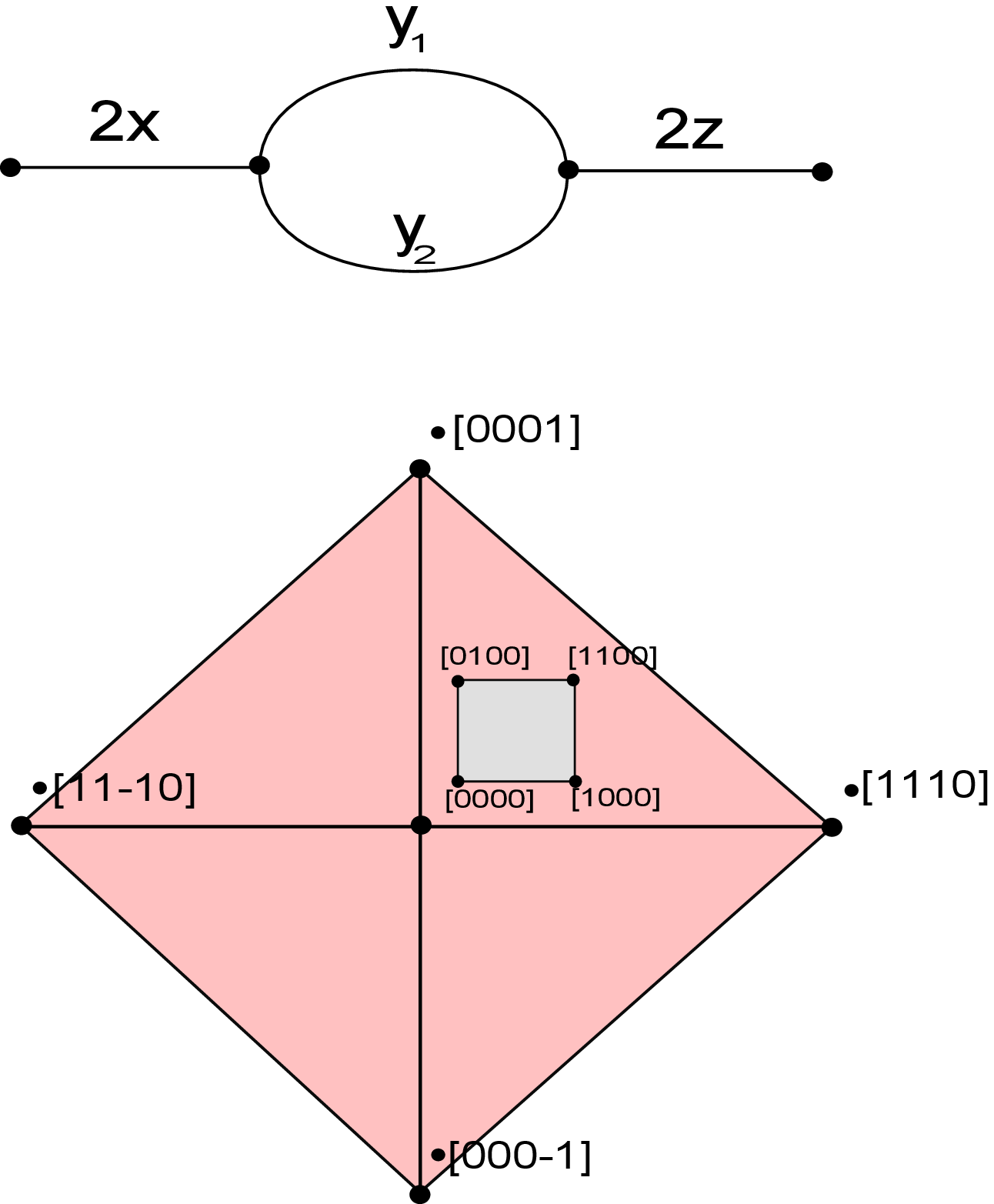}
\caption{The polytope $B_2(L),$ projected into the plane, the fibers over the corners are all points, and the fiber over the center is a square.}
\label{block2}
\end{figure}

The balanced polytope $P_3(L)$ corresponding to the weightings of an internal trinode has dimension $3.$   It is the convex hull of $(0,0,0), (L, L, 0), (L, 0, L),$ and $(0, L, L).$ For $L > 1,$ the standard regions of this polytope with respect to $\Sigma^2$ appear in \cite{M1}, Figure $3$.  For $L = 1,$ this polytope is a non-normal subpolytope of the unit square, this is the reason for the $L > 1$ condition in theorem \ref{polypres}.   It was shown in \cite{M1} that  each of the standard regions of $P_3(L)$ are normal, with quadratic generating relations, 
except for the region at the origin, the convex hull of $[0,0,0], [1, 1, 0], [1, 0, 1], [0, 1, 1],$ and  $[1, 1,1].$  
This polytope has one cubic relation, 

\begin{equation}
[0, 1, 1][1, 0, 1][1, 1, 0] = [0,0,0][1, 1, 1][1, 1, 1],\\ 
\end{equation}

\noindent
this is why Theorem \ref{polypres} stipulates that relations are generating by quadrics and cubics instead
of just quadrics.

Next we analyze the polytope $B_2(L).$  A lattice point of this polytope is given by $4$ non-negative integers $(2x, y_1, y_2, 2z)$
which satisfy the triangle inequalities and the parity condition.  The conditions defining weightings force the quantities $y_1 + y_2$ and $y_1 - y_2$  are forced to be even integers. 
We begin by subjecting this polytope to a change of coordinates.

\begin{equation}
A= \frac{y_1 - y_2}{2}\\
\end{equation}

\begin{equation}
B= \frac{y_1 + y_2}{2}\\
\end{equation}

\noindent
Under this transformation, $B_2(L)$ becomes the polytope on four numbers $x, z, A, B$ subject to the conditions,$x, z, B \geq 0$; $-x, -z \leq A \leq x, z$; $x, z \leq B \leq 2L$.
The projection of $B_2(L)$ onto the $A, B$ plane produces a quadrilateral, shown in Figure \ref{block2} with the fibers of the projection
depicted above each lattice point.  The polytope $B_2(L)$ has $\mathbb{Z}/2 \times \mathbb{Z}/2$ symmetry, which divides
it into four isomorphic quadrants. We represent these quadrants with interlacing diagrams on $4$ numbers below. Arrows in the diagrams point from smaller
entries to larger entries, in particular the polytope corresponding to an interlacing diagram with a single arrow $a \rightarrow b$ with entries bounded by $L$ is the simplex with vertices $[00], [0L], [LL]$ in $\R^2.$

\begin{figure}
$$
\begin{xy}
(-16, -16)*{Q_1(L)};
(0, -16)*{A};
(0, 16)*{B};
(-16, 0)*{x};
(16,0)*{z};
(0,-13)*{\bullet} = "A1";
(0,13)*{\bullet} = "B1";
(-13,0)*{\bullet} = "C1"; 
(13,-0)*{\bullet} = "D1";
"B1"; "C1";**\dir{-}? >* \dir{>};
"B1"; "D1";**\dir{-}? >* \dir{>};
"C1"; "A1";**\dir{-}? >* \dir{>};
"D1"; "A1";**\dir{-}? >* \dir{>};
(34, -16)*{Q_2(L)};
(50, -16)*{-A};
(50, 16)*{B};
(34, 0)*{x};
(66,0)*{z};
(50,-13)*{\bullet} = "A2";
(50,13)*{\bullet} = "B2";
(37,0)*{\bullet} = "C2"; 
(63,-0)*{\bullet} = "D2";
"B2"; "C2";**\dir{-}? >* \dir{>};
"B2"; "D2";**\dir{-}? >* \dir{>};
"C2"; "A2";**\dir{-}? >* \dir{>};
"D2"; "A2";**\dir{-}? >* \dir{>};
(-16, -66)*{Q_3(L)};
(0, -66)*{A};
(0,  -34)*{2L - B};
(-16, -50)*{x};
(16,-50)*{z};
(0,-63)*{\bullet} = "A3";
(0,-37)*{\bullet} = "B3";
(-13,-50)*{\bullet} = "C3"; 
(13,-50)*{\bullet} = "D3";
"B3"; "C3";**\dir{-}? >* \dir{>};
"B3"; "D3";**\dir{-}? >* \dir{>};
"C3"; "A3";**\dir{-}? >* \dir{>};
"D3"; "A3";**\dir{-}? >* \dir{>};
(34, -66)*{Q_4(L)};
(50, -66)*{-A};
(50, -34)*{2L - B};
(34, -50)*{x};
(66,-50)*{z};
(50,-63)*{\bullet} = "A4";
(50,-37)*{\bullet} = "B4";
(37,-50)*{\bullet} = "C4"; 
(63,-50)*{\bullet} = "D4";
"B4"; "C4";**\dir{-}? >* \dir{>};
"B4"; "D4";**\dir{-}? >* \dir{>};
"C4"; "A4";**\dir{-}? >* \dir{>};
"D4"; "A4";**\dir{-}? >* \dir{>};
\end{xy}
$$
\caption{Interlacing diagrams representing quadrants of $B_2(L).$}
\label{quad}
\end{figure}
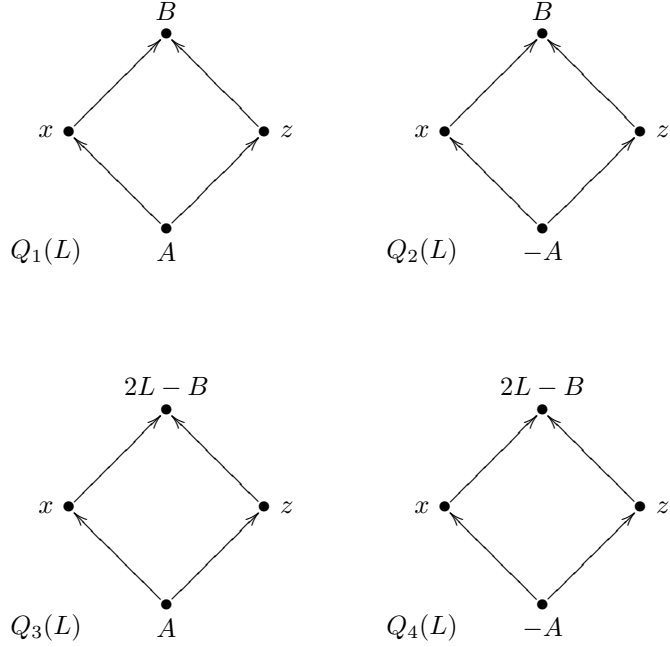

Every entry at the top of a diagram is less than or equal to $L$, and every element at the bottom is greater than or equal to $0.$
  
\begin{proposition}
The quadrant $Q_1(L)$ is a balanced polyope.
\end{proposition}

\begin{proof}
We establish that every element in a Minkowski sum $q \in Q_1(KL)$ has a balancing.  Consider the
interlacing diagram in the top left of the above diagram. For $K_1 + K_2 = K$ we form two new elements, one in $Q_1(k_1L)$ and the other in $Q_1(K_2L)$ by multiplying $q$ by $\frac{K_1}{K}$ (resp. $\frac{K_2}{K})$ and taking ceiling (resp. floor). 
Since $K_1$ and $K_2$ were arbitrary, we can repeat this process until we obtain $K$
elements in $Q_1(L).$  Note that this proves $Q(L)$ and therefore $B_2(L)$ are normal polytopes. 

Given two elements $q_1, q_2 \in Q_1(L)$ obtained from $q$ by this process, we can then form
the balancing $q_1', q_2'$ of $q_1 + q_2$ in $Q_1(2L)$ in the same way.  By Proposition \ref{balnum}, the resulting new factorization $q_1', q_2', \ldots$ of $q$ has $\Sigma^2$ weight
less than or equal to the $\Sigma^2$ weight of the original factorization, with equality occuring exactly
when the entries of $q_1, q_2$ are balanced with respect to each other.  This implies
that if a factorization of $q$ is not balanced, we may lower its $\Sigma^2$ weight with some
pairwise balancing, and after a finite number of these moves, we obtain a balanced factorization of $q.$
\end{proof}

\begin{remark}
The above proof can be adapted to any polytope defined by interlacing patterns. 
For example, this style of proof can be used to establish that the Gel'fand-Tsetlin polytope $GT(\lambda)$
defined by a dominant $SL_n(\C)$ weight $\lambda$ is balanced. 
\end{remark}

Since $B_2(L)$ is composed of $4$ isomorphic copies of $Q_1(L)$, it follows that it is a balanced polytope as well.  
We now introduce a term order on the lattice points of $B_2(L).$   For two lattice points, we first order by degree, 
then we order by the $\Sigma^2$ term-order. 
We complete this to a total term order as follows, $[x, z, A, B] < [x', z', A', B']$ iff 
 $B < B',$ or $B = B',$ and $z < z',$ or $B = B', z = z',$ and $x < x',$
or $B = B', z = z', x = x',$ and $z < z'.$  This induces a monomial order on polynomial ring $\C[X_{B_2(L)}],$
which surjects onto the semigroup algebra $\C[B_2(L)].$  

\begin{proposition}
 The term order defined above
induces a quadratic square-free Gr\"obner basis on the binomial ideal $I_{B_2(L)} \subset \C[X_{B_2(L)}].$
\end{proposition}

\begin{proof}
We sketch the proof.  First, by design, any standard region of the above term order will be a sub-set of a standard region of the $\Sigma^2$
ordering.   Note that we have shown that the standard regions of
this term order are the intersections of $B_2(L)$ with some integer translate $v + C_4$ of the unit cube. These
are all isomorphic to polytopes with entries between $0$ and $1,$ subject to the inequalities defined by some
sub-interlacing diagram of the defining diagram of $Q_1(1).$  A selection of these are depicted below. 

\begin{figure}
$$
\begin{xy}
(0, -16)*{A};
(0, 16)*{B};
(-16, 0)*{x};
(16,0)*{z};
(0,-13)*{\bullet} = "A1";
(0,13)*{\bullet} = "B1";
(-13,0)*{\bullet} = "C1"; 
(13,-0)*{\bullet} = "D1";
"B1"; "D1";**\dir{-}? >* \dir{>};
"C1"; "A1";**\dir{-}? >* \dir{>};
"D1"; "A1";**\dir{-}? >* \dir{>};
(40, -16)*{A};
(40, 16)*{B};
(24, 0)*{x};
(56,0)*{z};
(40,-13)*{\bullet} = "A2";
(40,13)*{\bullet} = "B2";
(27,0)*{\bullet} = "C2"; 
(53,-0)*{\bullet} = "D2";
"B2"; "C2";**\dir{-}? >* \dir{>};
"C2"; "A2";**\dir{-}? >* \dir{>};
(80, -16)*{A};
(80, 16)*{B};
(64, 0)*{x};
(96,0)*{z};
(80,-13)*{\bullet} = "A3";
(80,13)*{\bullet} = "B3";
(67,0)*{\bullet} = "C3"; 
(93,-0)*{\bullet} = "D3";
"B3"; "C3";**\dir{-}? >* \dir{>};
"C3"; "A3";**\dir{-}? >* \dir{>};
"B3"; "D3";**\dir{-}? >* \dir{>};
\end{xy}
$$
\caption{interlacing subdiagrams representing balanced regions of $B_2(L)$}
\label{subquad}
\end{figure}
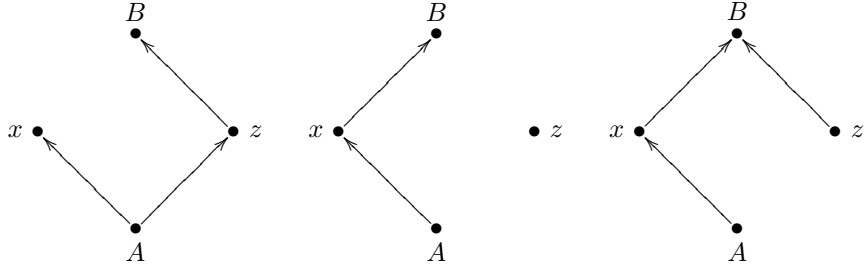

First we treat the polytope $Q(1).$  The semigroup defined by this polytope is generated
by six lattice points, $[1100], [1000], [0100], [0000], [0010], [1101],$ subject to one relation, 
$[11 00][0000] = [1000][0100].$  It follows that $Q(1)$ has a quadratic
square-free Gr\"obner basis with respect to the above term order. 

We now consider the polytopes corresponding to proper subdiagrams. Any such diagram has no loops, so it follows
that the corresponding balanced piece of $Q_i(L)$ can be represented, as in Proposition \ref{quadgrob}, as a fiber product
of simplices with term orders over unit intervals.  Furthermore, this fiber product can always be ordered in such a
way that the resulting term order agrees with the one defined above.   As a consequence of 
Proposition \ref{quadgrob}, each balanced piece of each $Q_i(L)$ has a quadratic, square-free
Gr\"obner basis with respect to the term order, and it follows that $I_{B_2(L)}$ has such a Gr\"obner basis as well. 
\end{proof}

Each of the above polytopes, $P_3(L)$ $B(L),$ $B_2(L),$ $P_3(s, r, L)$ and $P_3(r, L)$ come with distinguished
maps to the interval $[0, L]$ given by projecting onto the weight on a non-loop edge and dividing by $2.$  The polytope $P_3(L)$ has three such projections, $P_3(r, L),$ and $B_2(L)$ have two, and $B(L)$ and $P_3(s, r, L)$ each have one.  Each of these maps is given by forgetting coordinates off the edge, so they correspond to a coordinate projection.  If a monomial is balanced then it is easy to verify that any map which forgets coordinates gives another balanced monomial.  This implies that each of the edge projections above are morphisms in $\mathcal{P}$ to $([0, L], \Sigma^2).$   This allows us to prove Propositions \ref{polypres} and \ref{polyquad}.

\begin{proof}[Proof of Propositions \ref{polypres}, \ref{polyquad}]
In the first case, $P_{\Gamma}(\vec{r}, L)$ is a fiber product of balanced polytopes satisfying the stated conditions, therefore
$P_{\Gamma}(\vec{r}, L)$ inherits these conditions by Propositions \ref{flaggen}, \ref{flagrel}.

The second case is similar,  only we use Proposition \ref{quadgrob} to establish that $P_{\Gamma}(\vec{r}, L)$ inherits a quadratic, square-free
Gr\"obner basis. 
\end{proof}

\begin{figure}[htbp]
\centering
\includegraphics[scale = 0.45]{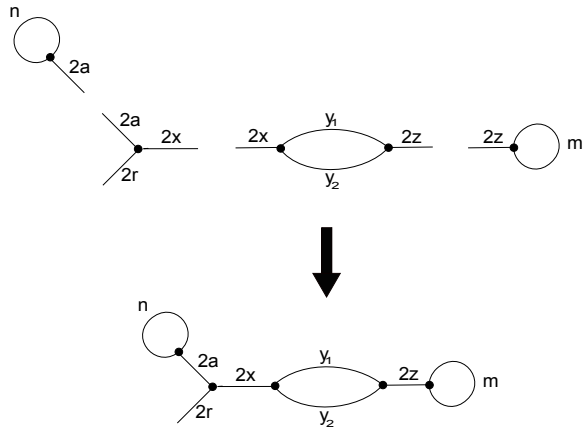}
\caption{Gluing building block polytopes with a fiber product.}
\label{assemble}
\end{figure}

\bigskip
\noindent
Christopher Manon:\\
Department of Mathematics,\\ 
University of California, Berkeley,\\ 
Berkeley, CA 94720-3840 USA,

\date{\today}

\end{document}